\documentclass[onefignum,onetabnum]{siamart171218}

\ifpdf
\hypersetup{
  pdftitle={Accelerated Optimization with Orthogonality Constraints},
  pdfauthor={Jonathan W. Siegel}
}
\fi

\usepackage{amsmath}
\usepackage{amssymb}
\usepackage{biblatex}
\usepackage{float}
\usepackage{xcolor,soul}
\usepackage{graphicx}
\usepackage{amsfonts}
\usepackage[algo2e]{algorithm2e}
\usepackage{placeins}
\usepackage[noend]{algpseudocode}
\let\oldref\ref
\renewcommand{\ref}[1]{(\oldref{#1})}
\DeclareMathOperator*{\argmin}{arg\,min}

\newcommand{\interior}[1]{%
  {\kern0pt#1}^{\mathrm{o}}%
}

\sethlcolor{white}

\title{Accelerated Optimization with Orthogonality Constraints\footnote{\MakeLowercase{This work previously appeared in the authors doctoral thesis}}}

\author{Jonathan W. Siegel\thanks{Department of Mathematics, Pennsylvania State University, University Park, PA 
  (\email{jus1949@psu.edu}, \email{jwsiegel2510@gmail.com}).}}

\begin{document}

\maketitle

\begin{abstract}
We develop a generalization of Nesterov's accelerated gradient descent method which is designed to deal with orthogonality constraints. To demonstrate the effectiveness of our method, we perform numerical experiments which demonstrate that the number of iterations scales with the square root of the condition number, and also compare with existing state-of-the-art \hl{quasi-Newton methods on the Stiefel manifold.} Our experiments show that our method outperforms existing state-of-the-art \hl{quasi-Newton} methods on some large, ill-conditioned problems.
\end{abstract}

\begin{keywords}Riemannian optimization, Stiefel manifold, accelerated gradient descent, eigenvector problems, electronic structure calculations\end{keywords}

\begin{AMS}65K05, 65N25, 90C30, 90C48\end{AMS}

\pagestyle{myheadings}
\thispagestyle{plain}
\markboth{TEX PRODUCTION}{USING SIAM'S \LaTeX\ MACROS}

\section{Introduction}
Optimization problems over the set of orthonormal matrices arise naturally in many scientific and engineering problems.
Most notably, eigenfunction and electronic structure calculations involve minimizing functions over the set of orthonormal matrices \cite{manifold_book, geometry, martin2004electronic, zhang2014gradient}.
In these applications, the objective functions are smooth but often ill-conditioned. There are also more recent applications which involve non-smooth
objectives, most notably the calculation of compressed modes \cite{ozolicnvs2013compressed}, which involve an $L^1$ penalization of
variational problems arising in physics.

In this paper, we consider optimization problems with orthogonality constraints, i.e. problems of the form
\begin{equation} \label{gen_prob}
    \argmin_{X^TX = I_k} f(X),
\end{equation}
where $X$ is an $n \times k$ matrix, $I_k$ is the identity matrix, and $f$ is a smooth function. The manifold of orthonormal matrices over which 
we are optimizing is referred to as the Stiefel manifold in the literature. Many methods have been proposed for solving \eqref{gen_prob}, including variants
of gradient descent, Newton's method, quasi-Newton methods, and non-linear conjugate gradient methods \cite{manifold_book,geometry,wen2013feasible,zhu2017riemannian,huang2015broyden}.
However, existing methods can suffer from slow convergence when the problem is ill-conditioned, by which we mean that the Hessian of $f$ at (or near) the minimizer is
ill-conditioned \cite{geometry}. Such problems are of particular interest, since they arise when doing electronic structure calculations, or when
solving non-smooth problems by smoothing the objective, for instance. Moreover, preconditioning such problems can be very difficult due to the
manifold constraint.

In an attempt to solve ill-conditioned problems more efficiently, we develop an extension of the well-known Nesterov's accelerated gradient descent algorithm \cite{nesterov1983method}
designed for optimizing functions on the Stiefel manifold. For the class of smooth, strongly convex functions on $\mathbb{R}^n$, accelerated gradient descent
obtains an asymptotically optimal iteration complexity of $O(\sqrt{\kappa})$, compared to $O(\kappa)$ for gradient descent with optimal step size selection (see \cite{bubeck2015convex}, section 3.7, here
$\kappa$ is the condition number of the problem). 
Our method extends this convergence behavior from $\mathbb{R}^n$ to the Stiefel manifold, thus providing an efficient method for solving ill-conditioned optimization problems
with orthogonality constraints.

Other work on accelerated gradient methods on manifolds includes \cite{liu2017accelerated} and \cite{zhang2018towards}. In 
\cite{liu2017accelerated} an accelerated gradient method on general manifolds is presented. However, their algorithm involves
solving a non-linear equation involving both the metric on the manifold and the objective function $f$. Unfortunately, solving this
equation is only feasible for the special type of model problem which they consider and cannot be generally applied to arbitrary
optimization problems on the Stiefel manifold. In \cite{zhang2018towards}, a theory is developed which shows that a certain type
of accelerated method can achieve accelerated convergence locally. However, their method involves calculating a geodesic logarithm
in every iteration and has not yet been implemented, although an iterative method for calculating the geodesic logarithm has
been developed in \cite{zimmermann2017matrix}. In constrast, our method only involves very simple linear algebra calculations
in each iteration and can be run efficiently on large problems.

The paper is organized as follows. In section 2, we briefly introduce the necessary notation and ideas from differential geometry. In section 3, we discuss accelerated gradient descent on $\mathbb{R}^n$. We recall results which are relevant
to our work. In section 4, we detail the design of our method. One of the key ingredients is an efficient procedure for performing approximate extrapolation and interpolation on the
manifold, which we believe could be useful in developing other optimization methods. In section 5, we show numerical results which provide evidence that our method achieves the
desired iteration complexity. Finally, in section 6, we present comparisons with other optimization methods on the Stiefel manifold. We show that our method outperforms existing state of the art methods on some large, ill-conditioned problems.

\section{Riemannian Manifolds}\label{manifolds-intro-section}
In this section, we briefly introduce the notation we will \hl{use} in the rest of the paper concerning the Stiefel manifold 
and differential geometry in general. We also collect some formulas for calculating on the Stiefel manifold which will be used
later. Some references for differential geometry include \cite{spivak1970comprehensive,lang2012fundamentals}
and for the geometry of the Stiefel manifold, see \cite{geometry}.

Let $M$ be a smooth manifold and $x\in M$. We denote the tangent space of $M$ at $x$ by $T_xM$ and the dual
tangent space by $(T_xM)^*$. We denote the tangent bundle of $M$by $TM$, and likewise the dual tangent bundle by $(TM)^*$.

Suppose $f$ is a $C^1$ function on $M$. We denote the derivative of $f$, by $\nabla f(x)\in (T_xM)^*$. \hl{Notice that the derivative of $f$ is naturally an element of the dual tangent space $(T_xM)^*$.}
If $M$ is a Riemannian manifold, then each tangent space $T_xM$ is equipped with a positive definite 
inner product $g:T_xM \times T_xM \rightarrow \mathbb{R}$. We denote the norm induced by $g$ as
$\|v\|^2_g$ and the norm induced by $g$ on the dual space as
$\|w\|_{g*}$.

Additionally, the inner product $g$ provides an isometry between the tangent space and its dual, which we denote by
$\phi_g:(T_xM)^*\rightarrow T_xM$ and $\phi^{-1}_g:T_xM\rightarrow (T_xM)^*$, and which are also known as raising
and lowering indices. \hl{Given a $C^1$ function $f$ on a Riemannian manifold $M$, an object which is often considered is the Riemannian gradient, obtained by raising the indices of the gradient $\nabla f(x)\in (T_xM)^*$ to obtain a tangent vector (instead of a dual tangent vector). In this work, we will not work with the Riemannian gradient explicitly, but will rather work with the true gradient $\nabla f(x)\in (T_xM)^*$, which we consider to be more natural. Whenever we need to raise or lower indices this will be explicitly stated.}

Assuming that the geodesic equations on $M$ can be solved globally in time (which is true for the Stiefel manifold that we are interested in)
we denote the exponential map based at $x\in M$ by $\exp_x:T_xM\rightarrow M$.

Let $x\in M$ and $f\in C^2(M)$. We denote the Hessian of $f$ at $x$ (viewed as a quadratic form on $T_xM$) by
\begin{equation}
    Hf(x)(v) = \left.\frac{d^2f(\exp_x(tv))}{dt^2}\right|_{t = 0}.
\end{equation}
The condition number of $Hf(x)$ will be important in what follows and is given by
\begin{equation}\label{Hessian_ratio}
    \kappa(Hf(x)) = \frac{\sup_{\|v\|_{g} = 1} Hf(x)(v)}{\inf_{\|v\|_{g} = 1} Hf(x)(v)}.
\end{equation}

In Riemannian optimization, geodesics are often expensive to compute exactly, which leads to the concept of a retraction.
A retraction associates to each point and tangent vector a curve on the manifold which approximates a geodesic to (at least) first order.
This means that if we only move a small distance, a retraction will be very close to the geodesic, even though a retraction can be vastly
different globally (see \cite{manifold_book} for examples and more discussion).
\begin{definition}
Let $M$ be a (smooth) manifold. A retraction on $M$ is a (smooth) map $R:TM\rightarrow M$ (here $TM$ denotes the
tangent bundle of $M$) satisfying for all $x\in M$ and $v\in T_xM$
\begin{equation}
    R(x,0) = x
\end{equation}
\begin{equation}
    \left.\frac{d}{dt}\right|_{t = 0} R(x, tv) = v
\end{equation}
(Here I write $R(x,v)$ for the image of the point $(x,v)\in TM$ under $R$.)
\end{definition}
We proceed to collect explicit formulas for each of these quantities on the Stiefel manifold.

\subsection{The Stiefel Manifold}
The Stiefel manifold $S_{n,k}$ is the set of $n\times k$ orthonormal matrices, i.e. 
$$S_{n,k} = \{X\in \mathbb{R}^{n\times k}: X^TX = I_k\}.$$
There are two metrics commonly put on the Stiefel manifold in the literature. One is obtained by
viewing \hl{$S_{n,k}\subset \mathbb{R}^{n\times k}$} and considering the metric induced by the ambient space \hl{$\mathbb{R}^{n\times k}$}.
The other, called the canonical metric and which we will be considering for the remainder of this paper, is obtained by viewing 
$S_{n,k} = O(n)/O(n-k)$ as the quotient of the orthogonal group $O(n)$ by the right action of
$O(n-k)$. Specifically, the action is given by right multiplication by
\begin{equation}
\begin{bmatrix}
I_{k\times k} & 0_{k\times n} \\
0_{n\times k} & O_{(n-k)\times (n-k)}
\end{bmatrix},
\end{equation}
where $O_{(n-k)\times (n-k)}\in O(n-k)$. This induces a quotient metric on $S_{n,k}$. For more details on the former metric and the differences between
these two viewpoints, see \cite{geometry}.

We fix the following representation of the elements of $S_{n,k}$, its tangent and
dual tangent spaces. The elements of $S_{n,k}$ will be represented by $n\times k$ orthonormal matrices (even though our metric is 
induced by viewing $S_{n,k}$ as a quotient $O(n)/O(n-k)$).

The tangent space at a point $X\in S_{n,k}$ is identified with the
set $T_X = \{V\in \mathbb{R}^{n\times k}:V^TX + X^TV = 0\}$.

We represent the dual space by
$(T_X)^* = \{W\in \mathbb{R}^{n\times k}:W^TX + X^TW = 0\}$, with the pairing between $T_X$ and $(T_X)^*$ given by $\langle V,W\rangle = \text{Tr}(V^TW)$ (i.e. the usual inner product on \hl{$\mathbb{R}^{n\times k}$}). Utilizing a slight abuse of notation, 
we can also interpret any matrix
\hl{$W\in \mathbb{R}^{n\times k}$} as an element of the dual space using this inner product. This corresponds to identifying the
matrix with its orthogonal projection onto the dual space $(T_X)^*$.

Using these representations, the quotient metric on $S_{n,k}$ is given by (see \cite{geometry})
\begin{equation} \label{inner_product_stiefel} 
    g(Y,Z) = \text{Tr}\left(Y^T\left(I - \frac{1}{2}XX^T\right) Z\right)
\end{equation}
and the inner product on the dual space is given by
\begin{equation}
    g^*(Y,Z) = \text{Tr}\left(Y^T\left(I + XX^T\right)Z\right).
\end{equation}

The maps corresponding to raising and lowering the indices are
\begin{equation}
    \phi_g(W) = \left(I + XX^T\right)W
\end{equation}
and
\begin{equation}
    \phi_g^{-1}(V) = \left(I - \frac{1}{2}XX^T\right)V.
\end{equation}

The advantage of using canonical metric, i.e. the metric induced by the quotient structure of $S_{n,k}$, is that geodesics can be computed
using the matrix exponential. In fact, the constant-speed geodesic starting at $X\in S_{n,k}$ and moving initially in the direction
$V\in T_XS_{n,k}$ is given by (see \cite{geometry}, \hl{equation 2.42, details in appendix} \ref{derivation-geodesics})
\begin{equation} \label{stiefel_exact_geodesics}
    X(t) = \exp{\left(t(VX^T - XV^T + XV^TXX^T)\right)}X.
\end{equation}
If we give our direction via a dual vector $W$ and raise indices first, we obtain the simpler expression
\begin{equation}
    X(t) = \exp{\left(t(WX^T - XW^T)\right)}X.
\end{equation}
This formula can be applied to any \hl{$W\in \mathbb{R}^{n\times k}$}. This is equivalent to viewing $W$ as an
element of the dual tangent space via the Frobenius inner product, i.e. orthogonally projecting $W$ onto the
dual tangent space. Notice, however, that we do not need to perform this projection explicitly, we can simply
insert $W$ into the above formula (see \cite{geometry,wen2013feasible}).
The matrix $WX^T - XW^T$ has rank \hl{at most} $2k$, and so this exponential can be calculated by diagonalizing a $2k\times 2k$
antisymmetric matrix, as shown in \cite{geometry}.

In \cite{wen2013feasible}, a different retraction is introduced, which can be viewed as a Pad\'e approximation of the
above exponential. Their retraction, called the Cayley retraction, is defined on a dual vector (or any arbitrary
\hl{$W\in \mathbb{R}^{n\times k}$} viewed as an element of the dual tangent space as mentioned above) by the formula
\begin{equation}
    R_C(X, \phi_g(W)) = \left(I - \frac{1}{2}(WX^T - X^TW)\right)^{-1}\left(I + \frac{1}{2}(WX^T - X^TW)\right)X,
\end{equation}
which can be calculated using the Sherman-Morrison-Woodbury formula \cite{sherman1950adjustment} as (with $U = [\frac{1}{2}W, X]$ and $Z = [X, -\frac{1}{2}W]$)
\begin{equation} \label{R1_retraction}
    R_1(X, \phi_g(W)) = X + 2U(I - Z^TU)^{-1}Z^TX.
\end{equation}
This retraction avoids the need to calculate a matrix exponential and reduces the number of floating point
operations required by a significant constant factor over the geodesic retraction (although both retractions
have the same asymptotic complexity, see \cite{wen2013feasible}).

\section{Accelerated Gradient Descent}
Let $f:\mathbb{R}^n\rightarrow \mathbb{R}$ be a differentiable convex function. We say that $f$ is $\mu$-strongly convex if
\begin{equation} \label{strong_convexity}
    f(y) \geq f(x) + \langle\nabla f(x), y - x\rangle + \frac{\mu}{2}\|x - y\|_2^2.
\end{equation}
We also say that $f$ is $L$-smooth if 
\begin{equation}
    f(y) \leq f(x) + \langle\nabla f(x), y - x\rangle + \frac{L}{2}\|x - y\|_2^2.
\end{equation}
One way of thinking about these definitions is that $\mu$-strong convexity implies that the eigenvalues of the Hessian
of $f$ at every point are greater than $\mu$ and $L$-smoothness implies that the eigenvalues are less than $L$.

In his seminal paper \cite{nesterov1983method}, Nesterov introduced first-order methods which achieves the asymptotically optimal 
objective error for the class of $L$-smooth convex functions and for the class of $L$-smooth and $\mu$-strongly convex functions. 
These methods take the form
\begin{equation}\label{accel_iter}
    x_0 = y_0,~x_{t+1} = y_t - \gamma_t \nabla f(y_t),~y_{t+1} = x_{t+1} + \alpha_t(x_{t+1} - x_t).
\end{equation}
The choice of $\gamma_t$ and $\alpha_t$ depend on whether the function $f$ is strongly convex (as opposed to only convex and
$L$-smooth), and also on the precise parameters $\mu$ and $L$.

If $f$ is $\mu$-strongly convex and $L$-smooth, then setting $\alpha_t = \frac{\sqrt{L} - \sqrt{\mu}}{\sqrt{L} + \sqrt{\mu}}$
and $\gamma_t = 1/L$ produces the asymptotically optimal objective error of $O((1 - \sqrt{\frac{\mu}{L}})^{-t})$ (compared with $O((1 - \frac{\mu}{L})^{-t})$ for
gradient descent), as the following theorem shows.
\begin{theorem}\label{strongly_convex_accelerated_convergence}
Assume that $f$ is $\mu$-strongly convex and $L$-smooth. Let $x^*$ be the minimizer of $f$. If we let $\alpha_t = \frac{\sqrt{L} - \sqrt{\mu}}{\sqrt{L} + \sqrt{\mu}}$
and $\gamma_t = 1/L$ in \eqref{accel_iter}, then we have that
\begin{equation}
    f(x_t) - f(x^*) \leq 2\left(1 - \sqrt{\frac{\mu}{L}}\right)^t(f(x_0) - f(x^*)).
\end{equation}
\end{theorem}
\begin{proof}
See, for instance, section 3.7 in \cite{bubeck2015convex}.
\end{proof}

One disadvantage of the method analyzed in Theorem \ref{strongly_convex_accelerated_convergence} is
that setting the proper step size and momentum parameter requires knowing the smoothness parameter $L$ and the strong convexity
parameter $\alpha$.

The optimal method for $L$-smooth functions is more flexible. In particular, no knowledge about the smoothness parameter
is needed. One can use a line search to determine the correct step size and still obtain the optimal objective error of $O(t^{-2})$ (compared with $O(t^{-1})$ for
gradient descent). In particular, we have the 
following result (which generalizes the results in \cite{su2016differential} to obtain a larger family of accelerated schemes).

\begin{theorem} \label{accelerated_convergence}
Assume that $f$ is convex and differentiable with minimizer $x^*$.

Let $q_t$ be any sequence of non-negative real numbers satisfying 
$q_0 = 0$ and $(q_{t + 1} + 1)^2 \leq (q_t + 2)^2 + 1$ (in particular $q_{t+1} \leq q_t + 1$ works).

Then, if in iteration \eqref{accel_iter}, $\gamma_t$ is chosen so that $\gamma_t \leq \gamma_{t-1}$ and $f(x_{t+1}) \leq f(y_t) - (\gamma_t/2)\|\nabla f(y_t)\|_2^2$, and 
$\alpha_t = \frac{q_t}{2 + q_{t+1}}$, we have
\begin{equation}
    f(x_t) - f(x^*) \leq 2(\gamma_tq_t(q_t + 2))^{-1}\|x_0 - x^*\|_2^2.
\end{equation}
\end{theorem}
\begin{proof}
See the appendix \ref{proof-accelerated-convergence}.
\end{proof}
Note that in the above theorem we made no assumption that $f$ was $L$-smooth. This emphasizes that the method is independent of the particular value of $L$. We choose the step size $\gamma_t$ to provide a sufficient decrease in the objective. Such a $\gamma_t$ can be found using a line search and will be about $1/L$ in the worst case (within a constant depending on the precise line search scheme). 

Also, setting $q_t = \alpha t$ and $\gamma_t = 1/L$ for $\alpha \leq 1$ recovers the result from \cite{su2016differential} (with $r = 1 + 2/\alpha$). In particular, the special case $\alpha = 1$ gives $f(x_t) - f(x^*) \leq 2Lt^{-2}\|x_0 - x^*\|_2^2$.

\subsection{Adaptive Restart for Strongly Convex Functions}
\hl{It is often the case in practice that an objective is strongly convex and smooth, but the strong convexity parameter $\mu$ and the
smoothness parameter $L$ are unknown. This creates a problem because the correct momentum and step size in Theorem} \ref{strongly_convex_accelerated_convergence} \hl{requires knowledge of the strong convexity and smoothness parameters. Many researchers have proposed methods for adaptively estimating the parameters $\mu$ and $L$ (see, for instance,} \cite{nesterov2007gradient}, \cite{lin2014adaptive} and \cite{o2015adaptive}).

\hl{These ideas are relevant to our work because we will in general not have knowledge of the local (i.e. near the minimizer) smoothness and strong convexity parameters for objectives on the Stiefel manifold, and we build upon the work presented in} \cite{o2015adaptive}. \hl{The methods introduced there are based on the following observation.}

Suppose we are given a $\mu$-strongly convex, $L$-smooth function $f$. Then since $f$ is convex and $L$-smooth, we can run iteration \eqref{accel_iter} with
the parameters given in Theorem \ref{accelerated_convergence} (setting $q_t = t$) and obtain the following objective error
\begin{equation}\label{accel_bound_non_strongly}
    f(x_t) - f(x^*) \leq 2Lt^{-2}\|x_0 - x^*\|_2^2.
\end{equation}
The strong convexity of $f$ now allows us to bound the iterate error by the objective error, since strong convexity implies that $(\mu/2)\|x_t - x^*\|_2^2 \leq f(x_t) - f(x^*)$.
Combining this with equation \eqref{accel_bound_non_strongly} we see that
\begin{equation}
    \|x_t - x^*\|_2^2 \leq 4(L/\mu)t^{-2}\|x_0 - x^*\|_2^2 = 4\kappa t^{-2}\|x_0 - x^*\|_2^2.
\end{equation}
where $\kappa = (L/\mu)$ is the condition number of $f$. This implies that after $t = \sqrt{8\kappa}$ iterations,
we will have
\begin{equation}
    \|x_t - x^*\|_2^2 \leq \frac{\|x_0 - x^*\|_2^2}{2}.
\end{equation}
So by restarting the method (i.e. setting $x_0 = x_t$ and resetting the momentum parameter) every $\sqrt{8\kappa}$ iterations,
we halve the iterate error every time we restart. This means that it takes $O(\sqrt{\kappa}\log(\epsilon))$ iterations to attain
an $\epsilon$-accurate solution and thus restarting the method at this frequency recovers the asymptotically optimal convergence rate (for
$\mu$-strongly convex $L$-smooth functions).

Of course, in order to apply this scheme, we must know the condition number $\kappa$ in order to determine the correct restart frequency. To get around this, the
method proposed in \cite{o2015adaptive} adaptively chooses when to restart based on an observable condition on the iterates. Specifically, they consider two
restart conditions
\begin{itemize}
    \item Function Restart Scheme: Restart when $f(x_t) > f(x_{t-1})$
    \item Gradient Restart Scheme: Restart when $\nabla f(y_{t-1})\cdot (x_t - x_{t-1}) > 0$
\end{itemize}

Both of these restart conditions are based upon the analysis of a quadratic objective and it is an open problem to fully analyze their
behavior when applied to an arbitrary strongly convex, smooth function. However, experimental results in \cite{o2015adaptive} show empirically
that the adaptively restarted methods perform well in practice.

\section{Acceleration on the Stiefel Manifold}
We now come to the heart of the paper. In this section, we develop a version of Nesterov's accelerated gradient descent which is
designed for efficiently optimizing functions on the Stiefel manifold. We do this by generalizing the adaptively restarted methods of the previous section to the Stiefel manifold. \hl{There are two main difficulties which we must overcome in doing this, the first of which results from the non-convexity of our objectives and is thus not specific to optimization on the Stiefel manifold.}

Our first difficulty is that the functions which we will be minimizing are non-convex. This is due to the fact that all globally convex functions on
the Stiefel manifold are constant \cite{yau1974non} (since the manifold is compact). Because of this, we cannot hope to obtain a global convergence rate.
Our goal is to construct a method which is guaranteed to converge and which will achieve an accelerated rate once it is close enough to the
(local) minimizer. \hl{Notice that this difficulty applies equally to non-convex optimization problems on other manifolds, including Euclidean space $\mathbb{R}^n$.}

The second difficulty, \hl{which is specific to the Stiefel manifold,} is that we must find an efficient way of generalizing the momentum step $$y_{t+1} = x_{t+1} + \alpha_t(x_{t+1} - x_t)$$ of \eqref{accel_iter}
to the manifold. We will develop a very efficient method for averaging and extrapolating on the Stiefel manifold, which can potentially be used to design a variety of other optimization methods as well.

In the following subsections, we will describe how each of these difficulties is overcome.

\subsection{Restart for Non-convex Functions}
When adapting accelerated gradient methods to the Stiefel manifold, we are faced with the issue that the manifold
is compact and so the only convex functions are constant. Consequently, the functions which we are optimizing are necessarily 
non-convex. In this case the convergence results of Theorems \ref{strongly_convex_accelerated_convergence} and \ref{accelerated_convergence} 
can not be applied and in fact we cannot hope for
a `global' convergence rate.

Instead, what we note is that in a small neighborhood of a local optimum $X^*$ the function will be strongly convex and
smooth, provided that the Hessian at $X^*$ is positive definite. Moreover, the ratio of the strong convexity and smoothness parameters
in this neighborhood will be the close to the condition number of $\nabla^2 f(X^*)$, which we denote by $\kappa(X^*)$.

Thus the accelerated gradient method analyzed in Theorem \ref{strongly_convex_accelerated_convergence} suggests that we should be able to find a method which achieves a convergence
rate of $O((1 - \kappa(X^*)^{-1/2})^t)$ once it is close enough to the local minimum $X^*$. But since we have to deal with
functions which are not globally convex, we hope to design a method which is
guaranteed to converge to a local minimum even for non-convex functions, and which achieves the optimal convergence rate
once it is close enough to the local minimum.

Our approach is to modify the function restart scheme considered in \cite{o2015adaptive} and described in the previous section. 
We introduce the following restart condition, which forces a sufficient decrease in the objective.
\begin{itemize}
    \item {Modified Function Restart Scheme: Restart when
    \begin{equation}\label{restart}
        f(x_{t+1}) > f(x_{t}) - c_R\gamma_t\|\nabla f(y_{t})\|^2,
    \end{equation}
    where $c_R$ is a parameter we take to be a small constant (recall that $\gamma_t$ is the step size at step $t$).}
\end{itemize}
The norm above is the norm of the gradient in the dual tangent space if we are on a Riemannian manifold and is taken to be the
Euclidean norm if we are considering $\mathbb{R}^n$. We will apply this condition to optimization problems on the Stiefel
manifold, but we analyze its convergence for non-convex functions on Euclidean space, 
\hl{a situation in which it applies equally well.}

\begin{theorem}
Let $f$ be a differentiable, $L$-smooth function on $\mathbb{R}^n$, i.e. 
$\nabla f$ is Lipschitz with constant $L$. Assume also that $f$ is bounded
below.

Consider the iteration \eqref{accel_iter} with step size $\gamma_t$ chosen to satisfy $c/L \leq \gamma_t \leq \gamma_{t-1}$ for some $c \leq 1$ 
and $f(x_{t+1}) \leq f(y_t) - (\gamma_t/2)\|\nabla f(y_t)\|_2^2$ (which can be done by choosing $\gamma_t\leq \frac{1}{L}$).

If this iteration is restarted whenever \eqref{restart} holds (with the new $\gamma_0$ chosen to be $\leq\gamma_t$), then we have
\begin{equation}
    \lim_{t\rightarrow \infty} \|\nabla f(x_t)\| \rightarrow 0.
\end{equation}
\end{theorem}
\begin{proof}
Note first that our condition on the step size $\gamma_t$ can always be satisfied, since by the $L$-smoothness of $f$
we have that $\gamma_t = c/L$ will always work.

Also note that since $x_0 = y_0$, the condition on the step size always guarantees that $f(x_1) \leq f(x_0) - c_R\gamma_0\|\nabla f(y_0)\|_2^2$.

So we can always run the algorithm \eqref{accel_iter} in a way which satisfies the conditions of the theorem.

To complete the proof, we note that the restart condition combined with the observation that we always take at least one step implies that
\begin{equation}
    f(x_{t+1}) \leq f(x_{t}) - c_R\gamma_t\|\nabla f(y_{t})\|_2^2.
\end{equation}
Summing this, we obtain
\begin{equation}
    c_R\displaystyle\sum_{t = 0}^N \gamma_t\|\nabla f(y_{t})\|_2^2 \leq f(x_0) - f(x_N).
\end{equation}
Since $f$ is bounded below, say by $M$ and $\gamma_n \geq c/L$ we see that
\begin{equation}
    \displaystyle\sum_{t = 0}^\infty \|\nabla f(y_{t})\|_2^2 \leq \frac{L(f(x_0) - M)}{c_Rc} < \infty.
\end{equation}
This implies that $\|\nabla f(y_t)\|\rightarrow 0$. Now we simply note that since $f$ is $L$-smooth and
$x_t = y_{t - 1} - \gamma_t\nabla f(y_{t-1})$, we have that
\begin{equation}
    \|\nabla f(x_t)\|_2 \leq (1 + L\gamma_t)\|\nabla f(y_{t-1})\|_2 \leq (1 + L\gamma_0)\|\nabla f(y_{t-1})\|_2
\end{equation}
where the last inequality is because $\gamma_t \leq \gamma_0$ by assumption. Thus, $\|\nabla f(x_t)\|\rightarrow 0$ as desired.
\end{proof}

\subsection{Extrapolation and Interpolation on the Stiefel Manifold}
We have now seen how to get around knowing the strong convexity and smoothness parameters and how to deal
with non-convex functions in the process. We proceed to address the second difficulty mentioned at the beginning of the section.
Namely, we consider the problem of generalizing the momentum step of \eqref{accel_iter}
\begin{equation} \label{momentum_step}
    Y_{t+1} = X_{t+1} + \alpha_t(X_{t+1} - X_t)
\end{equation}
to the manifold setting. 

More generally, we will consider the problem of efficiently extrapolating and interpolating on the Stiefel manifold,
i.e. given two points $X,Y\in S_{n,k}$ and $\alpha \in \mathbb{R}$, we want to calculate points $(1-\alpha)X + \alpha Y$ on a curve through $X$ and $Y$. By
setting $\alpha \in (0,1)$ this gives a way of averaging points on the manifold and by setting $\alpha > 1$ or $\alpha < 0$
we can extrapolate as in \eqref{momentum_step}.

A possible approach would be to perform the extrapolation or interpolation in Euclidean space
and then project back onto the Stiefel manifold. However, this projection step, which consists of taking the orthonormal part from
the polar decomposition of the matrix, is more expensive than the method we propose, especially for large matrices. One could also
replace the projection by a reorthogonalization procedure such as Gram-Schmidt (or a QR factorization). However,
this is quite inaccurate if $k$ (the number of vectors) is large and is also not as cheap as our method, which only involves
matrix products and inversions (but no factorizations).

The approach we take is both simpler and more efficient. What we propose for generalizing 
\begin{equation}
    (1-\alpha)X + \alpha Y
\end{equation}
is to solve for a $V\in (T_{X}S_{n,k})^*$ which satisfies (here $R$ is a retraction which we have fixed in the course of designing our method,
we present the corresponding equations in Euclidean space on the right to clarify the method)
\begin{equation} \label{vector_solve}
    Y = R(X, \phi_g(V)),~Y = X + V
\end{equation}
and to then extrapolate or average by setting 
\begin{equation}
    (1-\alpha)X + \alpha Y = R(X, \phi_g(\alpha V)),~(1-\alpha)X + \alpha Y = X + \alpha V.
\end{equation}
Note that the use of $\phi_g$ simply allows us to work in the dual tangent space.

The obvious difficulty with this is solving equation \eqref{vector_solve} for $V$, i.e. finding a $V$
such that $R(X,\phi_g(V)) = Y$ for some given $X$ and $Y$. However, if we take our retraction to be $R_1$ from the previous section (this is the
Cayley retraction introduced in \cite{wen2013feasible}), then this boils down to solving
\begin{equation}\label{extrapolation-equation}
    \left(I + \frac{1}{2}(VX^T - XV^T)\right)X = \left(I - \frac{1}{2}(VX^T - XV^T)\right)Y
\end{equation}
for $V$. Since $X^TX = Y^TY = I$, one can now easily check that $V = 2Y(I + X^TY)^{-1}$ solves this equation. Here
we are viewing $V\in\mathbb{R}^{nk}$ as an element of the dual tangent space via the Frobenius inner product, as mentioned in section
\ref{manifolds-intro-section}. \hl{In addition, we can check that if we replace $V$ by $V^\prime = V + XS$ for any symmetric $k\times k$ matrix $S$, then $VX^T - XV^T = V^\prime X^T - X(V^\prime)^T$. This means that $V^\prime$ also satisfies equation} \eqref{extrapolation-equation}. \hl{In particular, we can replace $V$ by its orthogonal projection onto the dual tangent space $(T_X)^*$, which then gives us the desired vector.}

One potential issue with this approach is the possibility that the matrix $(I + X^TY)$ could be singular or ill-conditioned.
To address this issue, we have the following lemma showing that the matrix $(I + X^TY)$ is well-conditioned as 
long as $X$ and $Y$ are not too far apart on $S_{n,k}$. We do not explicitly check this condition in our algorithms, but 
we have not experienced any numerical issues with the inversion of $I + X^TY$ in our experiments.

\hl{Before giving the result, we introduce some notation. Let $X,Y\in S_{n,k}$. We write the geodesic distance between $X$ and $Y$ with respect to the quotient metric as}
\begin{equation}\label{quotient-definition}
    d_{S_{n,k}^Q}(X,Y)^2 = \inf_{C(t):[0,1]\rightarrow S_{n,k}}\int_{0}^1 \text{Tr}\left(C^\prime(t)\left(I - \frac{1}{2}C(t)C(t)^T\right)C^\prime(t)^T\right)dt
\end{equation}
\hl{where the infemum is taken over all paths $C(t)$ which connect $X$ and $Y$, i.e. for which $C(0) = X$ and $C(1) = Y$.
Similarly, we write the geodesic distance with respect to the embedding metric as}
\begin{equation}\label{embedding-definition}
    d_{S_{n,k}^E}(X,Y)^2 = \inf_{C(t):[0,1]\rightarrow S_{n,k}}\int_{0}^1 \text{Tr}\left(C^\prime(t)C^\prime(t)^T\right)dt,
\end{equation}
\hl{where the infemum is taken over the same set where $C(0) = X$ and $C(1) = Y$.}

We now have the following result.
\begin{lemma}
    Suppose that $X,Y\in S_{n,k}$ and $d_{S_{n,k}^Q}(X,Y) < \sqrt{3/2}$. Then we have
    \begin{equation}
        \kappa(I + X^TY) \leq 2\left(3 - 2d_{S_{n,k}^Q}^2(X,Y)\right)^{-\frac{1}{2}}.
    \end{equation}
    Here $d_{S_{n,k}^Q}$ is the distance on the Stiefel manifold with the quotient metric.
\end{lemma}
\begin{proof}
We first note that
\begin{equation}\label{eq_589}
    \|X - Y\|_F \leq d_{S_{n,k}^E}(X,Y) \leq \sqrt{2}d_{S_{n,k}^Q}(X,Y),
\end{equation}
where $\|\cdot\|_F$ is the Frobenius norm. \hl{The
first inequality above is clear since the Frobenius norm is defined by the same infemum as in equation} \eqref{embedding-definition}, \hl{but without the restriction that the path $C(t)$ needs to lie on the Stiefel manifold. The second inequality follows trivially from equations} \eqref{quotient-definition} and \eqref{embedding-definition}, as noted in
\cite{geometry}, sections 2.2.1 and 2.3.1.

Now consider the condition number of the matrix $I + X^TY$
\begin{equation}\label{condition_number_bound}
    \kappa(I + X^TY) = \frac{\sup_{\|z\|_2 = 1} \|(I + X^TY)z\|_2}{\inf_{\|z\|_2 = 1} \|(I + X^TY)z\|_2}.
\end{equation}
Since $X$ and $Y$ are orthonormal, we clearly have $\|X^TYz\|_2 \leq 1$ and so the numerator above is bounded by $2$.
For the denominator, we note that
\begin{equation}\label{eq_603}
    \|(I + X^TY)z\|^2_2 = 1 + 2\langle Xz, Yz\rangle + \|X^TYz\|_2^2 \geq 1 + 2\langle Xz, Yz\rangle.
\end{equation}
We combine this with the fact that
\begin{equation}
    2\langle Xz, Yz\rangle = \|Xz\|_2^2 + \|Yz\|_2^2 - \|(X-Y)z\|_2^2.
\end{equation}
Since $X$ and $Y$ are orthonormal and $\|z\|_2 = 1$, we have that $\|Xz\|_2^2 = \|Yz\|_2^2 = 1$. Additionally, since the 
Frobenius norm bounds the operator norm, we have that $\|(X-Y)z\|_2^2 \leq \|X-Y\|_F^2$. Plugging this into equation
\eqref{eq_603} we see that
\begin{equation}
    \|(I + X^TY)z\|_2 \geq \sqrt{3 - \|X-Y\|_F^2}.
\end{equation}
Utilizing equation \eqref{eq_589} and \eqref{condition_number_bound}, we finally get
\begin{equation}
    \kappa(I + X^TY) \leq 2\left(3 - 2d_{S_{n,k}^Q}^2(X,Y)\right)^{-\frac{1}{2}}
\end{equation}
if $d_{S_{n,k}^Q}^2(X,Y) < \frac{3}{2}$.
\end{proof}

This gives us a computationally efficient procedure for averaging and extrapolating on the Stiefel manifold. We have already described how this
can be used to generalize accelerated gradient methods to the Stiefel manifold. We also propose
that this averaging and extrapolation procedure could potentially be a building block in other novel optimization algorithms on the manifold.

\subsection{Gradient Restart Scheme}
We can use the idea of the previous subsection to generalize the gradient restart scheme to the Stiefel manifold. Recall that the gradient restart scheme
restarts iteration \eqref{accel_iter} whenever 
$$\nabla f(Y_{t-1})\cdot (X_t - X_{t-1}) > 0.$$
We begin by noting that $X_t = Y_{t-1} - \gamma_{t-1}\nabla f(Y_{t-1})$ and so we can rewrite this condition as
\begin{equation}
    -\gamma_{t-1}\|\nabla f(Y_{t-1})\|_2^2 + \nabla f(Y_{t-1})\cdot (Y_{t-1} - X_{t-1}) > 0.
\end{equation}
Now it is clear that on the manifold $\|\nabla f(Y_{t-1})\|_2^2$ should become $\|\nabla f(Y_{t-1})\|_{g*}^2$. 
\hl{Here we are viewing
$\nabla f(Y_{t-1})$ as an element of the dual tangent space which it naturally is an element of, not as the Riemannian gradient which is obtained by raising the indices.} The tricky part is
generalizing $\nabla f(Y_{t-1})\cdot (Y_{t-1} - X_{t-1})$. What we propose is to solve for a $V\in (S_{Y_{t-1}})^*$ such that
\begin{equation} \label{solve_for_V}
    X_{t-1} = R(Y_{t-1}, \phi_g(V)).
\end{equation}
This element $V$ then serves as $X_{t-1} - Y_{t-1}$ and the analogue of the gradient restart condition becomes
\begin{equation}
    -\gamma_{t-1}\|\nabla f(Y_{t-1})\|_{g*}^2 - \langle\nabla f(Y_{t-1}), V\rangle_{g*} > 0.
\end{equation}
As in the previous subsection, we see that equation \eqref{solve_for_V} can be efficiently solved for $V$ if the retraction
we are using is $R_1$ (the Cayley retraction introduced in \cite{wen2013feasible}).

\subsection{Accelerated Gradient Descent on the Stiefel Manifold}
We now put together all of the the ideas presented in this section to obtain two versions of accelerated gradient descent
on the Stiefel manifold; the function restart variant, presented in \hl{algorithm} \ref{function_restart}, and the gradient restart scheme,
presented in algorithm \ref{gradient_restart}. In the next section we will present numerical results which demonstrate empirically that our methods
achieve the desired iteration complexity, i.e. that the number of iterations scales as $O(\sqrt{\kappa})$, 
where $\kappa$ is the condition number of the objective function.

\begin{algorithm}[!t]
\caption{Accelerated Gradient Descent with Function Restart Scheme}
\KwData{$f$ a smooth function, $\epsilon$ a tolerance, $\gamma_0$ an initial step size, $c_R$ a small restart parameter, $\lambda_d$, $c_L$ parameters required for the line search}
\KwResult{A point $X_t$ such that $\|\nabla f(X_t)\|_{g^*} < \epsilon$}
$X_0 \gets \text{initial point}$\\
$Y_0 \gets X_0$, $t \gets 0$, $k \gets 0$\\
\While {$\|\nabla f(X_t)\|_{g^*} \geq \epsilon$}{
    $X_{t+1} \gets R_1(Y_t, \phi_g(-\gamma_t\nabla f(Y_t)))$, perform a (two-sided) line search to ensure the Armijo condition $f(X_{t+1}) \leq f(Y_t) - \frac{1}{2}\gamma_t\|\nabla f(Y_t)\|_{g^*}^2$ is satisfied:\\ 
    \While {$f(X_{t+1}) < f(Y_t) - c_L\gamma_t\|\nabla f(Y_t)\|_{g^*}^2$}{
        $\gamma_t \gets \lambda_d\gamma_t$ \\
        $X_{t+1} \gets R_1(Y_t, \phi_g(-\gamma_t\nabla f(Y_t)))$
    }
    \While {$f(X_{t+1}) > f(Y_t) - \frac{1}{2}\gamma_t\|\nabla f(Y_t)\|_{g^*}^2$}{
        $\gamma_t \gets \gamma_t / \lambda_d$\\
        $X_{t+1} \gets R_1(Y_t, \phi_g(-\gamma_t\nabla f(Y_t)))$
    }
    \eIf {$f(X_{t+1}) > f(X_t) - c_R\gamma_t\|\nabla f(Y_t)\|_{g^*}^2$ (Restart Condition)} {
        $X_{t+1} \gets X_t$, $Y_t \gets X_{t+1}$, $k \gets 0$ (restart)\\
    }{
        $V_t \gets 2X_{t+1}(I + X_{t+1}^TX_t)^{-1}$\\
        $V_t \gets V_t - \frac{1}{2}X_t(V_t^TX_t + X_t^TV_t)$ (project onto the dual tangent space)\\
        $Y_{t+1} \gets R_1(X_t, (1 + \frac{k}{k+3})\phi_g(V_t))$ (apply momentum)\\
        $k\gets k+1$\\
    }
    $t\gets t+1$\\
    $\gamma_{t+1} = \gamma_t$\\
}\label{function_restart}
\end{algorithm}
\begin{algorithm}[!t]
\caption{Accelerated Gradient Descent with Gradient Restart Scheme}
\KwData{$f$ a smooth function, $\epsilon$ a tolerance, $\gamma_0$ an initial step size, $\lambda_d$, $c_L$ parameters required for the line search}
\KwResult{A point $X_t$ such that $\|\nabla f(X_t)\|_{g^*} < \epsilon$}
$X_0 \gets \text{initial point}$\\
$Y_0 \gets X_0$, $t \gets 0$, $k \gets 0$\\
\While {$\|\nabla f(X_t)\|_{g^*} \geq \epsilon$}{
    $X_{t+1} \gets R_1(Y_t, \phi_g(-\gamma_t\nabla f(Y_t)))$, perform a (two-sided) line search to ensure the Armijo condition $f(X_{t+1}) \leq f(Y_t) - \frac{1}{2}\gamma_t\|\nabla f(Y_t)\|_{g^*}^2$ is satisfied:\\ 
    \While {$f(X_{t+1}) < f(Y_t) - c_L\gamma_t\|\nabla f(Y_t)\|_{g^*}^2$}{
        $\gamma_t \gets \lambda_d\gamma_t$\\
        $X_{t+1} \gets R_1(Y_t, \phi_g(-\gamma_t\nabla f(Y_t)))$
    }
    \While {$f(X_{t+1}) > f(Y_t) - \frac{1}{2}\gamma_t\|\nabla f(Y_t)\|_{g^*}^2$}{
        $\gamma_t \gets \gamma_t / \lambda_d$\\
        $X_{t+1} \gets R_1(Y_t, \phi_g(-\gamma_t\nabla f(Y_t)))$
    }
    $W_t \gets 2X_t(I + X_t^TY_t)^{-1}$\\
    $W_t \gets W_t - \frac{1}{2}X_t(W_t^TX_t + X_t^TW_t)$ (projection onto the dual tangent space)\\
    \eIf {$\langle\nabla f(Y_t), W_t\rangle_{g*} < -\gamma_t\|\nabla f(Y_t)\|_{g^*}^2$ (Restart Condition)} {
        $X_{t+1} \gets X_t$, $Y_t \gets X_{t+1}$, $k \gets 0$ (restart)\\
    }{
        $V_t \gets 2X_{t+1}(I + X_{t+1}^TX_t)^{-1}$\\
        $V_t \gets V_t - \frac{1}{2}X_t(V_t^TX_t + X_t^TV_t)$ (projection onto the dual tangent space)\\
        $Y_{t+1} \gets R_1(X_t, (1 + \frac{k}{k+3})\phi_g(V_t))$ (apply momentum)\\
        $k\gets k+1$\\
    }
    $t\gets t+1$\\
    $\gamma_{t+1} = \gamma_t$\\
}\label{gradient_restart}
\end{algorithm}
\section{Numerical Results}
In this section, we provide the results of numerical experiments which test the convergence properties and robustness of the two algorithms described in the
previous section. We test the algorithms on a sequence of eigenvector calculations with increasing condition numbers. This allows us to investigate how the
iteration count scales with the condition number of the problem. The reason we do eigenvector calculations is that the condition number of the corresponding
objectives can be evaluated with relative ease.

\subsection{Single Eigenvector Calculations}
We begin by testing our algorithms on the sphere (which is a special case of the Stiefel manifold $S_{n,k}$ with $k = 1$). The problem we solve is the eigenvector calculation
\begin{equation}\label{numerical_eigenvector_opt}
    \argmin_{X\in S^n} \frac{1}{2}X^TAX,
\end{equation}
where $A$ is a symmetric matrix. The solution to this problem is the eigenvector corresponding to the smallest eigenvalue of $A$.

In order to evaluate the performance of our algorithm, we must investigate how the number of iterations scales with the condition number of \eqref{numerical_eigenvector_opt} (not to 
be confused with the condition number of $A$). A trivial calculation shows that this condition number is given by

\begin{equation}
    \kappa(Hf(v^*)) = \frac{\lambda_n - \lambda_1}{\lambda_2 - \lambda_1},
\end{equation}
where $\lambda_1,...,\lambda_n$ are the eigenvalues of $A$.

We choose a random initial point drawn uniformly at random on the sphere and apply both versions of our accelerated gradient descent
algorithm to optimize the objective \eqref{numerical_eigenvector_opt}. \hl{We take the matrix $A$ to be a diagonal matrix with diagonal entries (and eigenvalues) $d_{ii} = \lambda_i = i$.} Our stopping criterion is based on the relative gradient norm, i.e.
we stop when $\|\nabla f(X_n)\|_{g^*} \leq \epsilon \|\nabla f(X_0)\|_{g^*}$, with tolerance $\epsilon = 10^{-10}$. 
\hl{For both algorithms} \ref{function_restart} and \ref{gradient_restart} \hl{the initial step size is taken to be $\gamma_0 = 0.1$ and the line search parameters are taken as $\lambda_d = 1.7$ and $c_L = 0.7$. The restart parameter for algorithm} \ref{function_restart} \hl{is taken as $c_R = 0.01$. We remark that the performance of the algorithm is not particularly sensitive to these parameter values.}

We run this experiment for $21$ values of $n$ evenly spaced in log-space from $10^2$ to $10^4$ and plot the logarithm of the number of iteration
vs the logarithm of the condition number in figure \ref{scheme_test_2}. To reduce random fluctuations, we take the average of the logarithm
of the number of iterations over $50$ trials for each $n$. We also compare our results with a gradient descent scheme (using
 the same line search to determine the step size as our methods) and give the coefficients of a log-linear fit to the iteration data.

We see that our method empirically achieves the desired convergence behavior. Indeed, we can see that the number of iterations scales
slightly better than with the square root of the condition number, 
whereas the number of iterations for gradient descent scales about with the condition number. In addition, our accelerated method significantly
outperforms the gradient descent method even for relatively small condition numbers.
\begin{figure}[ht!]
\centering
\makebox[\textwidth][c]{\includegraphics[width=200mm]{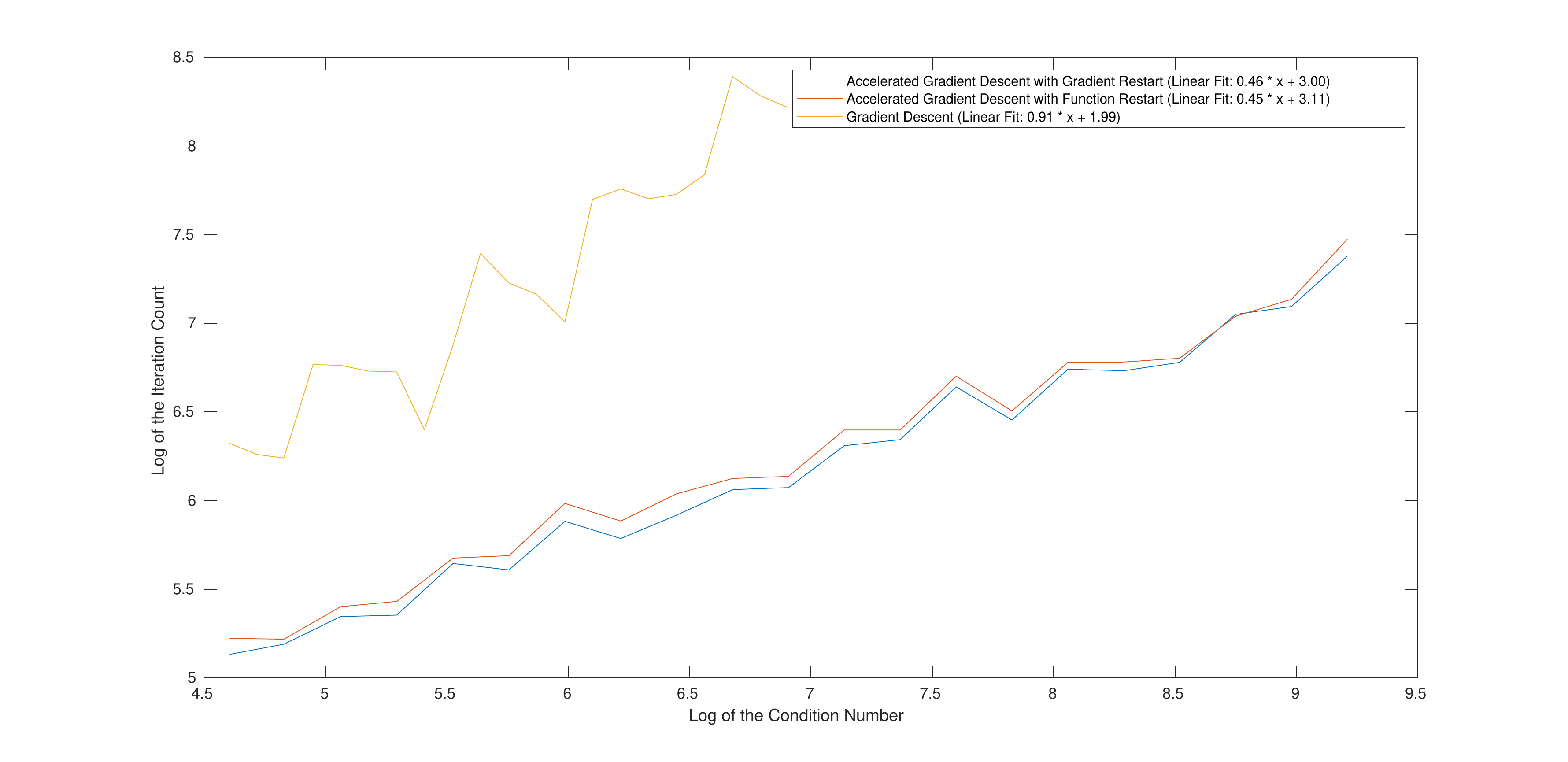}}
\caption{Iteration Count vs Condition Number (Sphere). We clearly see that the accelerated schemes exhibit better scaling than
simple gradient descent. The logarithm above is a natural logarithm (i.e. base $e$).}\label{scheme_test_2}
\end{figure}

\subsection{Multiple Eigenvector Calculations}
We now test our algorithms on the Stiefel manifold $S_{n,k}$ with $k > 1$. The problem we consider is that of calculating the smallest $k$ eigenvectors of a
symmetric linear operator $A$ by minimizing the Brockett cost
\begin{equation} \label{multiple_eigenvector_problem}
    \argmin_{X\in S_{n,k}} \frac{1}{2}\displaystyle\sum_{i = 1}^k \alpha_k \langle X_k, AX_k\rangle,
\end{equation}
where $X_k$ denotes the $k$-th column of $X$ and $0 < \alpha_1 < \alpha_2 < ... < \alpha_k$ are coefficients which force the minimizer to consist of
eigenvectors of $A$ rather than eigenvectors up to an orthogonal transformation.

As before, we want to investigate how the number of iterations depends upon the condition number of \eqref{multiple_eigenvector_problem}. A
simple calculation shows that the condition number of the Brockett cost \eqref{multiple_eigenvector_problem} is 

\begin{equation}\label{condition_number_calc_eigenvectors}
    \kappa(Hf(X^*)) = \frac{\alpha_k(\lambda_n - \lambda_1)}{\min\{\alpha_1(\lambda_{k+1} - \lambda_k), \min_{i < k} (\lambda_{k-i+1} - \lambda_{k-i})(\alpha_{i+1} - \alpha_i)\}}.
\end{equation}

We briefly note that if one knew the eigenvalues $\lambda_1,...,\lambda_n$, then the coefficients $\alpha_1,...,\alpha_n$ which minimize the
condition number \eqref{condition_number_calc_eigenvectors} produce a minimial condition number of

\begin{equation}\label{optimal_condition_number}
    \kappa(Hf(X^*))_{opt} = (\lambda_n - \lambda_1)\left(\displaystyle\sum_{i = 1}^k \frac{1}{\lambda_{i+1} - \lambda_i}\right).
\end{equation}
\hl{Indeed, this minimum condition number is achieved for}
\begin{equation}\label{optimal_alpha}
    \alpha_i = \displaystyle\sum_{j=0}^{i-1} \frac{1}{\lambda_{k-i+1} - \lambda_{k-i}},
\end{equation}
\hl{or any multiple of this choice.}

As with the calculations on the sphere, we choose a random initial point on $S_{n,k}$ and utilize the same relative gradient norm
stopping criterion with a tolerance $\epsilon = 10^{-10}$. \hl{We again let $A$ be a diagonal matrix with $d_{ii} = \lambda_i = i$. We take for the weights $\alpha_i$ the condition number minimizing choice corresponding to these eigenvalues, which is simply $\alpha_i = i$. We use the same value for the initial step size $\gamma_0$ and the line search and restart parameters $\lambda_d, c_L, c_R$ as for the single eigenvector calculations. We remark that as before the performance of the algorithm is not particularly sensitive to these parameter values.}

As before, we run this experiment for $21$ values of $n$ evenly spaced in log-space from $10^2$ to $10^4$ and plot the logarithm of the number of iteration
vs the logarithm of the condition number in figure \ref{scheme_test_10_2}. To reduce random fluctuations, we take the average of the logarithm
of the number of iterations over $50$ trials for each $n$. We also compare our results with a gradient descent scheme (using
 the same line search to determine the step size as our methods) and give the coefficients of a log-linear fit to the iteration data.

We again see that the gradient restart scheme empirically achieves the
desired convergence behavior. Indeed, we see that the number of iterations scales
better than with the $\sqrt{\kappa}$, but that gradient descent achieves a scaling close to $\kappa$. As before, our methods also
significantly outperform gradient descent for all problems tested.

\begin{figure}[ht!]
\centering
\makebox[\textwidth][c]{\includegraphics[width=200mm]{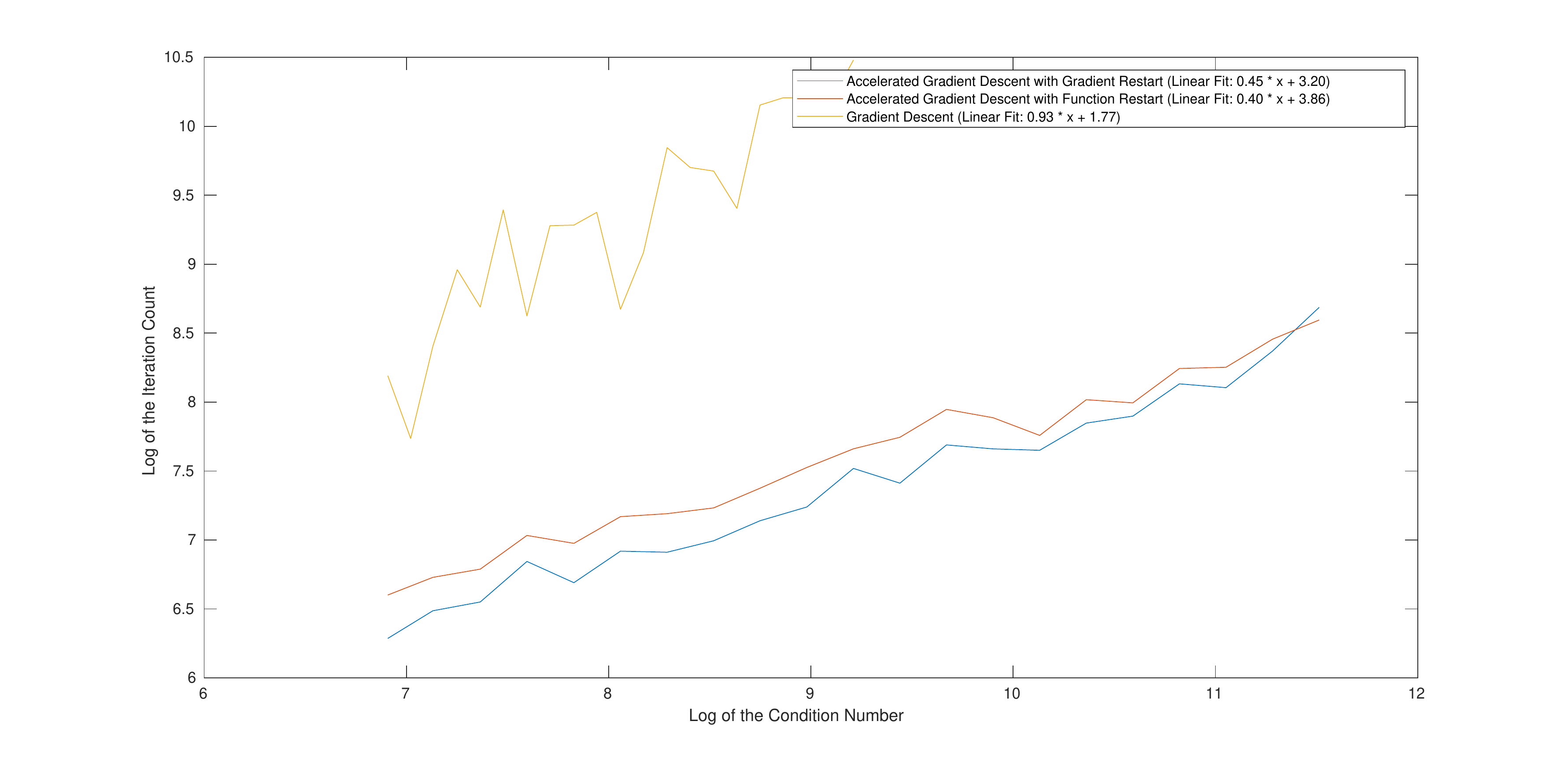}}
\caption{Iteration Count vs Condition Number ($k = 10$). We see that the accelerated method scales much better than ordinary gradient descent. The logarithm above is a natural logarithm (i.e. base $e$).}\label{scheme_test_10_2}
\end{figure}

\subsection{Comparison with State-of-the-Art Quasi-Newton Methods}
In this section, we compare our algorithm with an existing state of the art method.
We only consider our function restart scheme since it performed better than the gradient restart scheme in our previous tests. 
We compare with the state of the art L-RBFGS quasi-Newton method implemented in the ROTLIB library \cite{huang2018roptlib} (see also
\cite{huang2015broyden}).

We run both methods against each other on an ill-conditioned Brockett cost function \eqref{multiple_eigenvector_problem}, 
\hl{where the matrix $A$ is taken to be diagonal with spectrum $\{\frac{k^2}{n}\}_{k = 1,...,n}$ 
and the weights are $\alpha_i = i$ (note that this is not the optimal choice of $\alpha_i$ given in equation }\eqref{optimal_alpha}, \hl{however, this is not particularly important as the main point is that this objective is ill-conditioned).} In order to obtain 
reliable results, we test both methods with the same objective and the same random initial points. We repeat each test $10$ times
and record the average number of iterations, number of gradient evaluations, number of function evaluations, and total computational
time. We use a relative gradient norm stopping condition, i.e.
we stop when $\|\nabla f(X_n)\|_{g^*} \leq \epsilon \|\nabla f(X_0)\|_{g^*}$, with tolerance $\epsilon = 10^{-9}$ (we use the
slightly larger tolerance of $10^{-9}$ because the line search implemented in ROPTLIB failed when optimizing further). \hl{The parameters we use for the function restart scheme} \ref{function_restart} \hl{are initial step size $\gamma_0 = 0.1$, line search parameters $\lambda_d = 1.7$ and $c_L = 0.9$, and restart parameters $c_R = 0.01$.}

In our tests we consider two different values of $n$ and $k$. First, we test with $n = 1000$, $k = 10$, whose results are shown in 
table \ref{small_comparison}. Then we test with the larger values $n = 2000$, $k = 20$. The results of this test are shown in table
\ref{larger_comparison}. We see that our method outperforms the method in \cite{huang2015broyden}, using about a third
as many gradient evaluations, about the same number of functions evaluations, and less than half the time as the quasi-newton method
for the larger problem $n = 2000$, $k = 20$.

\begin{table}[h!]
    \centering
    \begin{tabular}{|c|c|c|c|c|}
    \hline
         Method & Iterations & Function Evals & Gradient Evals & Time (s) \\
    \hline
         Quasi-Newton & 31027.4 & 32528.9 & 31028.4 & 26.3 \\
    \hline 
         Accelerated Gradient & 17266.2 & 43513.4 & 17267.2 & 17.8 \\
    \hline
    \end{tabular}
    \caption{Comparison of the methods on a problem of size $n = 1000$, $k = 10$.}\label{small_comparison}
\end{table}

\begin{table}[h!]
    \centering
    \begin{tabular}{|c|c|c|c|c|}
    \hline
         Method & Iterations & Function Evals & Gradient Evals & Time (s) \\
    \hline
         Quasi-Newton & 84767.2 & 89013.8 & 84768.2 & 169.4 \\
    \hline 
         Accelerated Gradient & 28758.8 & 93747.8 & 28759.8 & 64.5 \\
    \hline
    \end{tabular}
    \caption{Comparison of the methods on a problem of size $n = 2000$, $k = 20$.}\label{larger_comparison}
\end{table}

\section{Conclusion} 
In this paper, we developed novel accelerated first-order optimization methods designed to handle orthogonality constraints. The algorithms developed 
are a generalization of Nesterov's gradient descent to the Stiefel manifold. In the process, we constructed an efficient way of averaging and extrapolating points on the
manifold, which we believe can be useful in developing other novel optimization algorithms. Numerical experiments indicate that our methods 
not only achieve the desired
scaling with the condition number of the problem, but also outperform state of the art \hl{quasi-Newton} methods on some large, ill-conditioned problems.

We would also like to note that although the algorithms we have constructed make explicit use of formulas specific to the Stiefel manifold,
we believe the ideas presented in this paper can be generalized to other manifolds as well. This would depend upon generalizing the
momentum step \eqref{momentum_step} efficiently to the manifold of interest. In fact, for quotients of the Stiefel manifold, for example
the Grassmann manifold, the algorithms developed in this paper are already immediately applicable.

\section{Acknowledgements}
We are grateful to Russel Caflisch, Stanley Osher, and Vidvuds Ozolins for suggesting this project and providing continued guidance. This work was supported by
AFOSR grant FA9550-15-1-0073.

\printbibliography

\appendix
\section{Proof of Theorem \ref{accelerated_convergence}}\label{proof-accelerated-convergence}
\begin{proof}[Proof of Theorem \ref{accelerated_convergence}]
Consider the Lyapunov function
\begin{equation}
    J_t = \gamma_tq_t(q_t + 2)(f(x_t) - f(x^*)) + \frac{1}{2}\|2(y_t - x^*) + q_t(y_t - x_t)\|_2^2.
\end{equation}
We will show that $J_{t+1} \leq J_t$ which proves the theorem since $\gamma_tq_t(q_t + 2)(f(x_t) - f(x^*)) \leq J_t$
and $J_0 = 2\|y_0 - x^*\|_2^2 = 2\|x_0 - x^*\|_2^2$. To this end, we denote
\begin{equation}
    J^1_t = \gamma_tq_t(q_t + 2)(f(x_t) - f(x^*))
\end{equation}
and
\begin{equation}
    J^2_t = \frac{1}{2}\|2(y_t - x^*) + q_t(y_t - x_t)\|_2^2.
\end{equation}
Then we see that
\begin{equation}
\begin{split}
    &J^1_{t+1} - J^1_t = \gamma_tq_t(q_t + 2)(f(x_{t+1}) - f(x_t)) + \\(\gamma_{t+1}&q_{t+1}(q_{t+1} + 2) - \gamma_tq_t(q_t + 2))(f(x_{t+1}) - f(x^*)).
\end{split}
\end{equation}
Since by assumption $\gamma_{t+1} \leq \gamma_t$ and $q_{t+1}(q_{t+1} + 2) = (q_{t+1} + 1)^2 - 1 \leq (q_t + 2)^2$, we see that
$\gamma_{t+1}q_{t+1}(q_{t+1} + 2) \leq \gamma_{t}(q_{t} + 2)^2$, and the bottom line in the above equation is bounded by
\begin{equation}
    (\gamma_{t}(q_{t} + 2)^2 - \gamma_tq_t(q_t + 2))(f(x_{t+1}) - f(x^*)) = 2\gamma_t(q_t + 2)(f(x_{t+1}) - f(x^*)).
\end{equation}
Thus we see that
\begin{equation}
    J^1_{t+1} - J^1_t \leq \gamma_t(q_t + 2)[2(f(x_{t+1}) - f(x^*)) + q_t(f(x_{t+1}) - f(x_t))].
\end{equation}
The step sizes $\gamma_t$ are chosen so that $f(x_{t+1}) - f(y_t) \leq -(1/2)\gamma_t\|\nabla f(y_t)\|_2^2$, and so we can rewrite the above to
obtain
\begin{equation}
\begin{split}
    J^1_{t+1} - J^1_t& \leq \gamma_t(q_t + 2)[2(f(y_t) - f(x^*)) + q_t(f(y_t) - f(x_t))]\\
    &-\frac{(\gamma_t(q_t + 2))^2}{2}\|\nabla f(y_t)\|_2^2.
\end{split}
\end{equation}
The convexity of $f$ implies that $f(y_t) - f(x^*) \leq \nabla f(y_t)\cdot(y_t - x^*)$ and $f(y_t) - f(x_t) \leq \nabla f(y_t)\cdot(y_t - x_t)$, so we get
\begin{equation}
\begin{split}
    J^1_{t+1} - J^1_t& \leq \gamma_n(q_t + 2)\nabla f(y_t)\cdot[2(y_t - x^*) + q_t(y_t - x_t)]\\
    &-\frac{(\gamma_t(q_t + 2))^2}{2}\|\nabla f(y_t)\|_2^2.
\end{split}
\end{equation}
Now we consider $J^2_{t+1} - J^2_t$. Note that $J^2_t = (1/2)\|s_t\|_2^2$ with $$s_t = 2(y_t - x^*) + q_t(y_t - x_t).$$
Thus 
\begin{equation}
J^2_{t+1} - J^2_t = (s_{t+1} - s_t)\cdot s_t + \frac{1}{2}\|(s_{t+1} - s_t)\|_2^2.
\end{equation}
Considering that 
\begin{equation}
    J^1_{t+1} - J^1_t \leq \gamma_t(q_t + 2)\nabla f(y_t)\cdot s_t-\frac{(\gamma_t(q_t + 2))^2}{2}\|\nabla f(y_t)\|_2^2,
\end{equation}
we will be done if we can show that $s_{t+1} - s_t = -\gamma_t(q_t + 2)\nabla f(y_t)$. To this end we compute
\begin{equation}\label{eq_836}
\begin{split}
    s_{t+1} - s_t& = 2(y_{t+1} - y_t) + q_t(y_{t+1} - y_t) - q_t(x_{t+1} - x_t) \\
    &+ (q_{t+1} - q_t)(y_{t+1} - x_{t+1}).
\end{split}
\end{equation}
Recalling the update formulas for the iterates $x_n$ and $y_n$, we see that $$y_{t+1} - x_{t+1} = \alpha_t(x_{t+1} - x_t)$$ and $$y_{t+1} - y_t = -\gamma_t\nabla f(y_t) + \alpha_t(x_{t+1} - x_t).$$ Thus equation \eqref{eq_836}
simplifies to
\begin{equation}
    s_{t+1} - s_t = -(q_t + 2)\gamma_t\nabla f(y_t) + (2\alpha_t + q_{t+1}\alpha_t - q_t)(x_{t+1} - x_t),
\end{equation}
which is equal to $-(q_t + 2)\gamma_t\nabla f(y_t)$ by our choice of $\alpha_t$.
\end{proof}

\section{Derivation of equation \ref{stiefel_exact_geodesics}}\label{derivation-geodesics}
We begin from equation 2.42 in \cite{geometry}, which, using their notation is
\begin{equation}
    Y(t) = Qe^{Xt}I_{n,p},
\end{equation}
where
\begin{equation}
    X = \begin{bmatrix}
A & -B^T \\
B & 0_{(n-p)\times (n-p)}
\end{bmatrix}
\end{equation}
for some $A\in\mathbb{R}^{p\times p}$ skew-symmetrix and some $B\in \mathbb{R}^{n-p\times p}$.
We rewrite this using the fact that $Q$ is an orthonormal matrix as
\begin{equation}
    Y(t) = \exp{(tQXQ^T)}QI_{n,p}.
\end{equation}

Now, translating into the notation we have used, we see that $Q = [X,X^\prime]$, $Y(t) = X(t)$, and
$QXI_{n,p} = V$ (here $X^\prime$ is a basis for the orthogonal complement of $X$). Equation \ref{stiefel_exact_geodesics} now boils down to checking that (in our notation)
\begin{equation}
    [X,X^\prime]\begin{bmatrix}
A & -B^T \\
B & 0_{(n-p)\times (n-p)}
\end{bmatrix}[X,X^\prime]^T = VX^T - XV^T + XV^TXX^T
\end{equation}
where $V = XA+X^\prime B$ and $X^TX^\prime = 0$ since $X^\prime$ is an orthogonal complement of $X$.
\end{document}